\documentclass[12pt]{amsart}

\usepackage{amsmath,amssymb}
\usepackage{amsfonts}
\usepackage{amsthm}
\usepackage{latexsym}
\usepackage{graphicx}
\usepackage{epsfig}

\usepackage[T1]{fontenc}
\usepackage{times}
\usepackage{mathptmx}

\theoremstyle{plain}
\newtheorem{theorem}{Theorem}
\newtheorem{lemma}[theorem]{Lemma}

\theoremstyle{remark}

\newtheorem{remark}{Remark}[section]
\newtheorem{definition}{Definition}[section]

\renewcommand{\leq}{\leqslant}
\renewcommand{\geq}{\geqslant}

\newcommand{\inner}[2]{\langle #1,#2\rangle}

\newcommand{\R}{{\mathbb R}}

\renewcommand{\div}{\mathop{\mathrm{div}}}

\DeclareMathOperator{\Ric}{Ric}%
\DeclareMathOperator{\Rm}{Rm}%
\DeclareMathOperator{\Vol}{Vol}%

\begin{document}

\title{On Locally Conformally
Flat Gradient Shrinking Ricci Solitons}

\author{Xiaodong Cao$^*$}
\thanks{$^*$Research
partially supported by the Jeffrey Sean Lehman Fund from Cornell
University}
\address{Department of Mathematics,
  Cornell University, Ithaca, NY 14853}
\email{cao@math.cornell.edu, wang@math.cornell.edu}

\author{Biao Wang}
\thanks{}

\author{Zhou Zhang
}

\address{Department of Mathematics, University of Michigan,
at Ann Arbor, MI 48109}
\email{zhangou@umich.edu}
\renewcommand{\subjclassname}{%
  \textup{2000} Mathematics Subject Classification}
\subjclass[2000]{Primary 53C25, 53C44}

\date{Nov. 8th, 2008}

\begin{abstract}  In this paper, we first apply an integral identity
on Ricci solitons to prove that closed locally conformally flat
gradient Ricci solitons are of constant sectional curvature. We
then generalize this integral identity to complete noncompact
gradient shrinking Ricci solitons, under the conditions that the Ricci
curvature is bounded from below and the Riemannian curvature
tensor has at most exponential growth. As a consequence of this
identity, we classify complete locally conformally flat
gradient shrinking Ricci solitons with Ricci curvature bounded from below.
\end{abstract}

\maketitle

\markboth{Xiaodong Cao, Biao Wang and Zhou Zhang} {On Locally
Conformally Flat Gradient Shrinking Ricci Solitons}

\section{Introduction and Main Theorem}

The study of Ricci soliton has been an important part in the study of the
Ricci flow. It usually serves as a (dilation) limit of solutions to the
Ricci flow. Though being a Ricci soliton is a purely static condition,
it is usually convenient to view it as a
solution to the Ricci flow.

\begin{definition} A complete n-dimensional Riemannian manifold $(M, g)$ is
called a Ricci soliton if
\begin{equation}
     2{}R_{ij}+\nabla_{i}V_{j}+\nabla_{j}V_{i}=2\rho{}g_{ij}\ ,
   \end{equation}
for some vector field $V$
and constant $\rho$, where $R_{ij}$ is the Ricci curvature tensor.
Moreover, if $V$ is a gradient vector field of a function $f$ on
$M$, then we have a gradient Ricci soliton,  satisfying the
equation
   \begin{equation}\label{eq:grs}
      R_{ij}+\nabla_{i}\nabla_{j}f=\rho{}g_{ij}\ .
   \end{equation}
We say that $(M,g)$ is expanding, steady or shrinking
if $\rho$ is $<0$, $=0$ or $>0$, respectively. Notice that when
$f$ is a constant function, we have the Einstein equation. \end{definition}

The classification of gradient Ricci solitons has been a very
interesting problem. For closed expanding and steady gradient
Ricci solitons, it is well-known that they must be Einstein (see
\cite{per02} or \cite{CCGGIIKLLN}, Proposition 1.13). The
shrinking case is a little bit more complicated. When $n=2, 3$, it
is known that closed shrinking Ricci solitons are Einstein (see
\cite{H} and \cite{I}). If $n>3$, there exist shrinking Ricci
solitons which are not Einstein (See  N.
Koiso \cite{K}, H.-D. Cao \cite{CHD1996}, M. Feldman, T. Ilmanen and D. Knopf \cite{FIK}).
The recent work of C. B\"ohm and B. Wilking (\cite{BW}) implies that
gradient Ricci solitons with positive curvature operator must be
of constant curvature.

The classification of complete noncompact gradient Ricci solitons has
been studied by many authors very recently under various conditions,
for example, see
\cite{naber07}, \cite{NW2007}, \cite{pw07} and \cite{fmz08}.

In \cite{CXD2007}, the first author proves several identities on
closed gradient Ricci solitons, we now extend one of them to the case of
complete noncompact gradient Ricci solitons. Since the proof of
this identity requires integration by parts, we need to have
some control on  curvatures and the potential function $f$ such
that we can justify  integration by parts. Our main theorem is the
following:

\begin{theorem}\label{lem:C2007-(1.4)-B}
Let $(M,g)$ be a non-flat complete noncompact shrinking gradient
Ricci soliton given by \eqref{eq:grs}, suppose that the Ricci curvature
of $(M,g)$ is bounded from below. Assume further that it satisfies
\begin{equation}\label{eq:NW2007-(1.1)}
   |R_{ijkl}|(x)\leq\exp(a(r(x) + 1))
\end{equation} for some constant $a>0$, where $r(x)$ is the distance
function to a fixed
point on the manifold.
 Then the identity \begin{equation}\label{eq:cao2007(jga)-(1.11)}
      \int_{M}|\nabla\Ric|^{2}e^{-f} d\mu=
      \int_{M}|\div\Rm|^{2}e^{-f} d\mu
   \end{equation} holds.
\end{theorem}

As an application of the above integral identity
(\ref{eq:cao2007(jga)-(1.11)}), we will consider gradient shrinking Ricci solitons
(GSRS) which are
locally conformally flat. Such classification has also been
considered in \cite{elm06}, \cite{NW2007} and \cite{pw07}.
We will then
prove the following theorem for closed Ricci solitons.

\begin{theorem}\label{main theorem 01}
Let $(M,g)$ be a closed gradient shrinking Ricci soliton. Assume
that $(M,g)$ is locally conformally flat. Then $(M, g)$ must be
Einstein. Moreover, $(M, g)$ is of constant sectional curvature.
\end{theorem}

\begin{remark}
In this case, the problem has also been studied by M. Eminenti,
G. La Nave
and C. Mantegazza \cite{elm06} using a different method.
\end{remark}

In the complete noncompact case, the identity (\ref{eq:cao2007(jga)-(1.11)})
again yields a
classification of locally conformally flat GSRSs with
Ricci curvature bounded from below.

\begin{theorem}\label{main theorem 02}
Let $(M\sp n,g)$, $n\geq 3$, be a complete noncompact gradient
shrinking soliton whose Ricci curvature is bounded from below.
Assume that $(M,g)$ is locally conformally flat. Then its universal cover is either
$S^{n-1}\times\R$ or $\R^{n}$.
\end{theorem}


\begin{remark}
In the case of locally conformally flat GSRSs, the condition of
(\ref{eq:NW2007-(1.1)}) is same as $$|R_{ij}|(x)\leq\exp(a(r(x) +
1)).$$ This condition is not needed in Theorem 3.
\end{remark}

\begin{remark} L. Ni and N. Wallach (\cite{NW2007})
proved a slightly different version of this theorem. P. Petersen
and W. Wylie (\cite{pw07}) also proved a similar result. In dimension $3$,
their results do not require that $M$ to be locally conformally flat.

After we have finished the proof of Theorem \ref{main theorem 01},
we saw that L. Ni and N. Wallach have just proved their theorem,
from which we realize that our proof can be extended to the complete
noncompact case. The
approach here is different.
\end{remark}

\begin{remark}
An early version of this paper by the first two authors appeared
on arXiv July 2008 (arXiv:0807.0588v2). Most recently, L. Chen and
W. Chen \cite{cc08} informed us that they can improve our result based
on a similar idea.
\end{remark}

The rest of this paper is organized as follows. In Section 2, we
will prove Theorem 2 for closed gradient Ricci solitons. The proof
makes use of an integral identity obtained by the first author
\cite{CXD2007}. In Section 3, we will deal with the complete
noncompact case. We first prove the above mentioned integral
identity on complete gradient Ricci solitons (Theorem 1), then we
will finish the proof of Theorem 3.

{\bf Acknowledgement:} The authors would like to thank Professor
Lei Ni for his interest and encouragement for us to finish the
writing of this paper.

\section{Compact Case}
In this section, we will deal with the case of closed Ricci
solitons and prove Theorem \ref{main theorem 01}. We first recall
the following lemma from \cite[Corollary 1]{CXD2007}. Note that
the proof only uses the fact that $(M,g)$ is a closed gradient
soliton.

\begin{lemma}\label{lem:cao2007-(1.4)}
Suppose that $(M,g)$ is a closed gradient Ricci soliton, then we have
   \begin{equation}\label{eq:closed}
      \int_{M}|\nabla\Ric|^{2}e^{-f} d\mu=
      \int_{M}|\div\Rm|^{2}e^{-f} d\mu\ .
   \end{equation}
\end{lemma}

As $(M,g)$ is locally conformally flat, we have
   \begin{align}\nonumber \label{rmc}
      R_{ijkl}=&\,\frac{1}{n-2}\,(R_{ik}g_{jl}+R_{jl}g_{ik}
                  -R_{il}g_{jk}-R_{jk}g_{il})\\
               &\,-\frac{R}{(n-1)(n-2)}\,
                  (g_{ik}g_{jl}-g_{il}g_{jk})\ .
   \end{align}
Moreover, the following identity holds (see
\cite{E}, Eq. (28.19)),
   \begin{equation}\label{eq:chow1992(cpam)-(2.5)}
       \nabla_{k}R_{ij}-\nabla_{j}R_{ik}=\frac{1}{2(n-1)}\,
       (\nabla_{k}Rg_{ij}-\nabla_{j}Rg_{ik})\ .
   \end{equation}
On a closed Riemannian manifold $(M,g)$, it follows from the
second Bianchi identity that
\begin{align}
      (\div\Rm)_{jkl}
          &=\nabla_{i}R_{ijkl}=\nabla_{i}R_{klij}
           =-\nabla_{k}R_{ijli}-\nabla_{l}R_{ijik}\notag\\
           \label{eq:cao2007(jga)-(2.1)}
          &=\nabla_{k}R_{jl}-\nabla_{l}R_{jk}\ .
   \end{align}
As a consequence of \eqref{eq:chow1992(cpam)-(2.5)} and
\eqref{eq:cao2007(jga)-(2.1)}, on a closed locally conformally
flat gradient Ricci soliton, we arrive at
   \begin{equation*}
        |\div\Rm|^{2}=|\nabla_{k}R_{ij}-\nabla_{j}R_{ik}|^{2}
        =\frac{1}{2(n-1)}\,|\nabla{}R|^{2}\ .
   \end{equation*}
The above identity and \eqref{eq:closed} now imply that
   \begin{equation}\label{eq:dRC-ge-dR}
        \int_{M}|\nabla\Ric|^{2}e^{-f}d\mu=
        \frac{1}{2(n-1)}\int_{M}|\nabla{}R|^{2}e^{-f}d\mu\ .
   \end{equation}
Since we have
   \begin{equation*}
      0\leq\left|\nabla_{k}R_{ij}-
      \frac{1}{n}\,\nabla_{k}Rg_{ij}\right|^{2}=
      |\nabla\Ric|^{2}-\frac{1}{n}\,|\nabla{}R|^{2}\ ,
   \end{equation*}
this implies
   \begin{equation*}
      \frac{1}{n}\,|\nabla{}R|^{2}\leq{}|\nabla\Ric|^{2}\ .
   \end{equation*}
Plugging this into \eqref{eq:dRC-ge-dR}, we obtain the following
inequality
   \begin{equation*}
      \frac{1}{n}\int_{M}|\nabla{}R|^{2}e^{-f}d\mu \leq
      \frac{1}{2(n-1)}\int_{M}|\nabla{}R|^{2}e^{-f}d\mu \ .
   \end{equation*}
But this is only possible if $n\leq{}2$ or $\nabla{}R\equiv{}0$.
Since we have already assumed that $n\geq{}3$, we must have
$\nabla{}R=0$. Hence that the scalar
curvature $R$ is constant. Take trace on both sides of
\eqref{eq:grs}, then we have
\begin{equation}\label{eq:grs (scalar case)}
   R+\Delta{}f=n\rho\ ,
\end{equation}
take integration on both sides of \eqref{eq:grs (scalar case)},
we have
\begin{equation*}
   (n\rho-R)\Vol(M)=\int_{M}\Delta{}fd\mu=0\ ,
\end{equation*}
here we use the fact that $M$ is closed. Therefore $R=n\rho$, and
then $\Delta{}f=0$. So  $f$ must be a constant. From
\eqref{eq:grs}, we know that $(M,g)$ is Einstein, i.e.,
\begin{equation*}
   R_{ij}=\frac{R}{n}\,g_{ij}\ .
\end{equation*}
Plugging this into \eqref{rmc} implies that
\begin{equation*}
   R_{ijkl}=\frac{R}{n(n-1)}\,(g_{ik}g_{jl}-g_{il}g_{jk})\ ,
\end{equation*}
and $R$ is a constant, so $(M, g)$ is of constant curvature. This
finishes the proof of Theorem \ref{main theorem 01}.

\section{Complete Noncompact Case}

In this section, we will first extend Lemma
\ref{lem:cao2007-(1.4)}  to the case of complete noncompact
gradient Ricci solitons (Theorem \ref{lem:C2007-(1.4)-B}). We will
prove that equality (\ref{eq:cao2007(jga)-(1.11)}) is also true
for a complete gradient shrinking soliton whose Ricci curvature is
bounded from below. Then we will finish the proof of Theorem
\ref{main theorem 02}.

Before we prove Theorem \ref{lem:C2007-(1.4)-B}, we will need the
following lemmas.

\begin{lemma}[\cite{CXD2007}]Suppose $(M,g)$ is a
complete gradient Ricci soliton, then we have the identities
\begin{equation}\label{eq:cao2007-(2.2)}
   \nabla_{i}(R_{ijkl}e^{-f})=0,
\end{equation}
and
\begin{equation}\label{eq:cao2007-(2.3)}
   \nabla_{i}(R_{ik}e^{-f})=0\ .
\end{equation}
\end{lemma}

\begin{lemma}\label{lem:NW2007-(2.3)}
Assume the same hypothesis as in Theorem \ref{lem:C2007-(1.4)-B}.
For any $D>0$, there exist constants $B>0$ and $C>0$ such that
\begin{equation}\label{eq:NW2007-(2.4)}
  f(x)\geqslant min\{\frac{\rho}{2}r(x)^2, \frac{Dr(x)}{3\rho}\},
\end{equation}
and
\begin{equation}
   f(x)\leq{}Cr^{2}(x)\ ,\quad
   |\nabla{}f|(x)\leq{}Cr(x)
\end{equation}
for $r(x)\geq{}B$. Here $r(x)$ is the distance function to some
fixed point $O\in{}M$, and $\rho$ is the same constant given by
\eqref{eq:grs}.
\end{lemma}

\begin{proof}
We consider a complete noncompact gradient shrinking soliton $(M, g)$. In
other words, over $M$, for some constant $\rho>0$, we have
\begin{equation*}
\nabla^2f+{\rm Ric}-\rho g=0.
\end{equation*}
We also assume that  ${\rm Ric} \geqslant -Kg$ for some $K>0$. In
the following $C$'s might stand for different positive constant,
but they are uniformly chosen.

Fix a point $O\in M$. For any $x\in M$, set $dist_g(O, x)=r(x)=
s_0$. Consider a minimal geodesic, $\gamma(s)$, from $O$
to $x$ with arc-length, $s$, as parameter. In the following, we'll
identify the arc-length parameter with the point on $M$. \\

We only need to consider points whose distance to $O$ is large
enough. We will choose this constant $B$ uniformly as in the
statement of the lemma.

To begin with, we have $\nabla^2f\leqslant Cg$ by soliton equation
(\ref{eq:grs}) and our assumption on the lower bound of Ricci
curvature. Thus $f(x)\leqslant Cr(x)^2$ for $x$ with $r(x)$ large
enough. Moreover, from $R+|\nabla f|^2-2\rho f=0$ \footnote{This
equation can be derived from the original Ricci soliton equation
using second Bianchi identity, where we have normalized the right
hand side to be $0$.} and the lower bound of scalar curvature $R$
(from lower bound of Ricci curvature), we have
$$|\nabla f|^2\leqslant Cr^2.$$

\vspace{0.2in}

It's only left to show lower bound of the growth for $f$ as
(\ref{eq:NW2007-(2.4)}). Let $\{\gamma', E_1, \cdots,E_{n-1}\}$
be a parallel orthonormal basis along this geodesic. For $s_0>2$
and $r_0$ which is relatively small and will be chosen later,
consider the following variation vector field:
\begin{equation*}
Y_i(s)=
\begin{cases}
sE_i(s), ~~~s\in [0, 1], \\
E_i(s), ~~~s\in [1, s_0-r_0], \\
\frac{s_0-s}{r_0}E_i, ~~~ s\in [s_0-r_0, s_0].
\end{cases}
\end{equation*}
By standard second variation formula of the minimal geodesic, we
have
$$\sum_{i=1}^{n-1} \int_0^{s_0} \left(|\nabla_{\gamma'}Y_i|^2-R
(\gamma', Y_i, \gamma', Y_i)\right)ds\geqslant 0,$$ or
\begin{equation*}
\begin{split}
0\leqslant
&\int_0^1\left(n-1-s^2\cdot {\rm Ric}(\gamma', \gamma')\right)
ds+\int_1^{s_0-r_0}\left(-{\rm Ric}(\gamma', \gamma')\right) ds+ \\
&~~~~ +\int_{s_0-r_0}^{s_0}\left(\frac{n-1}{r_0^2}-\left(\frac{s_0-s}
{r_0}\right)^2{\rm Ric}(\gamma', \gamma')\right) ds.
\end{split}
\end{equation*}

We can reformulate it in the following way,
\begin{equation*}
\begin{split}
\int_0^{s_0-r_0}{\rm Ric}(\gamma',\gamma')ds
&\leqslant n-1+\int_0^1(1-s^2){\rm Ric}(\gamma', \gamma')
ds+\frac{n-1}{r_0}+ \\
&~~~~ -\int_{s_0-r_0}^{s_0}\left(\frac{s_0-s}{r_0}\right)^2
{\rm Ric}(\gamma', \gamma')ds.
\end{split}
\end{equation*}
Using Ricci lower bound on the last term of right hand side, and
noticing the uniform bound of the metric in the unit ball around
the fixed point $O$, one arrives at
\begin{equation}
\label{eq:ricci-integral}
\int_0^{s_0-r_0}{\rm Ric}(\gamma',\gamma')ds\leqslant C+\frac
{n-1}{r_0}+Cr_0.
\end{equation}

Our next step is to control the integral of ${\rm Ric}(\gamma',
\gamma')$ between $s_0-r_0$ and $s_0$. Let's consider the
following two cases.

\vspace{0.2in}

\begin{itemize}

\item {\bf Case 1:} If there exists a uniform constant $D>0$
(which will eventually be chosen to be large enough in our
application), such that the scalar curvature at $x$ satisfies
$$R(x)\leqslant Ds_0,$$
then following the argument as in \cite[Proposition 1.1]{ni05}, we
can achieve the desired bound as follows.

Using soliton equation again, we get $\nabla_iR=2R_{ij}f_j$, hence
\begin{equation}
\label{eq:gradient-scalar}
|\nabla R|^2=4|R_{ij}f_j|^2.
\end{equation}

By choosing diagonal forms of ${\rm Ric}$ and $g$, using the lower
bound on Ricci curvature, we have
$$R_{ii}\geqslant -K, ~~~R+C\geqslant R_{ii}.$$
Hence (\ref{eq:gradient-scalar}) implies that
$$|\nabla R|^2\leqslant 4(R+C)^2|\nabla f|^2.$$
It follows that
$$|\nabla\log(R+C)|^2\leqslant 4|\nabla f|^2 \leqslant
Cr^2$$ from previous gradient bound, hence we have
$|\nabla\log(R+C)| \leqslant Cr$.

Now for any $s_1\in [s_0-r_0, s_0]$, we have
$$\log\left(\frac{R(s_1)+C}{R(s_0)+C}\right)\leqslant
\int_{s_1}^{s_0}|\nabla\log(R+C)|ds\leqslant Cs_0 (s_0-s_1).$$ So
we arrive at,
\begin{equation}\label{eq:end-sc}
R(s_1)+C\leqslant (R(s_0)+C) e^{Cs_0(s_0- s_1)}\leqslant
(R(s_0)+C) e^{Cs_0r_0}.
\end{equation}

Now let's pick $r_0$ such that
$$s_0r_0=\frac{n-1}{\epsilon},$$
where $\epsilon$ is a sufficiently small but fixed positive
constant.  It is clear that $r_0\leqslant
C$ for large $s_0$. \\

The integral of Ricci curvature between $s_0-r_0$ and $s_0$ can
now be controlled as follows,
\begin{equation}
\begin{split}
\int_{s_0-r_0}^{s_0}{\rm Ric}(\gamma', \gamma')ds
&\leqslant \int_{s_0-r_0}^{s_0}(R+C)ds \\
&\leqslant r_0(R(s_0)+C)e^{\frac{C(n-1)}{\epsilon}}  \\
&\leqslant r_0(Ds_0+C)e^{\frac{C(n-1)}{\epsilon}} \\
&\leqslant C. \nonumber
\end{split}
\end{equation}
Here we used the assumption on $R(s_0)$ and (\ref{eq:end-sc}) in
the above estimate. Combining this with (\ref{eq:ricci-integral}),
we show that, for small $\epsilon>0$,
$$\int_0^{s_0}{\rm Ric}(\gamma',\gamma')ds\leqslant
C+\frac{n-1}{r_0}+Cr_0+C\leqslant C
+\epsilon s_0.$$

Then one can deduce a (convex) quadratic lower bound
for $f(x)$ by
\begin{equation*}
\begin{split}
\gamma'(f)(x)-\gamma'(f)(O)
&= \int_0^{s_0}\nabla^2(f)(\gamma', \gamma')ds \\
&= \int_0^{s_0}(\rho-{\rm Ric(\gamma', \gamma')})ds \\
&\geqslant  (\rho-\epsilon)s_0-C,
\end{split}
\end{equation*}
which proves that $f(x)\geqslant\frac{\rho}{2}r(x)^2$ for $x$ with
large $r(x)=s_0$.

\vspace{0.2in}

\item {\bf Case 2:} If $R(x)\geqslant Ds_0$, then as $R+
|\nabla f|^2-2\rho f=0$, we have
$$f=\frac{1}{2\rho}(R+|\nabla f|^2)\geqslant \frac{Ds_0}
{2\rho}.$$

\end{itemize}

In conclusion, we have, for $x$ with large $r(x)$,
$$f(x)\geqslant min\{\frac{\rho}{2}r(x)^2, \frac{Dr(x)}{2\rho}\}.$$

The proof of the lemma is thus finished.

\end{proof}

Now we can finish the proof of Theorem \ref{lem:C2007-(1.4)-B}.

\begin{proof}[Proof of Theorem \ref{lem:C2007-(1.4)-B}]
The first author proved the identity
\eqref{eq:cao2007(jga)-(1.11)} when $M$ is a closed shrinking
gradient soliton in \cite[P427 $\sim$ P429]{CXD2007}. The key
point of the proof is using integration by parts (IBP). If we can
show that all of the IBPs in \cite[P427 $\sim$
P429]{CXD2007} are valid for the soliton $(M,g)$  in Theorem \ref{lem:C2007-(1.4)-B}, then we
can repeat the proof.

The integrations by parts we need to check are listed as follows:
\begin{align}\label{eq:IBP-(01)}
   \int_{M}R_{lkjp}f_{p}\nabla_{k}R_{jl}e^{-f}d\mu
      &=-\int_{M}R_{lkjp}f_{pk}R_{jl}e^{-f}d\mu\ ,\\
   \int_{M}\nabla_{k}R_{jl}\nabla_{l}R_{jk}e^{-f}d\mu
      &=\int_{M}R_{jk}(\nabla_{j}R_{ik}f_{i}-
        \nabla_{i}\nabla_{j}R_{ik})e^{-f}d\mu\ ,\\
   \int_{M}R_{jk}\nabla_{j}R_{ik}f_{i}e^{-f}d\mu
      &=-\int_{M}R_{jk}R_{ik}f_{ij}e^{-f}d\mu\ ,\\
   \int_{M}\nabla_{j}\nabla_{i}R_{ik}R_{jk}e^{-f}d\mu
      &=-\int_{M}\nabla_{i}R_{ik}\nabla_{j}(R_{jk}e^{-f})d\mu\ ,
\end{align}
where $f_{p}=\nabla_{p}f$, $f_{pk}=\nabla_{p}\nabla_{k}f$, etc.
We will just check \eqref{eq:IBP-(01)}. The others are similar.

Before we check the above identities, we need the derivative
estimate of curvatures. We now view our shrinking soliton as a
solution to the Ricci flow on $t \in [-1,0]$ with $g(0)=g$ and
\begin{equation}\label{eq:NW2007-(1.1)-02}
   |R_{ijkl}|(y,t)\leq{}C\exp\left(\frac{3}{2}\,ar(x)\right)\ ,
   \quad\forall\,(y,t)\in{}
   B_{g(-1)}\left(x,\frac{r(x)}{2}\right)\times[-1,0],
\end{equation}
here $r(x)$ is the distance function to some fixed point $O\in{}M$
with respect to the metric $g(0)$, $C$ is some constant depends only
on $n$. Applying the local derivative
estimate of W.-X. Shi \cite{shi89} (cf. \cite{H1995}), we have
\begin{align*}
   |\nabla\Rm|(x,0)&\leq{}C\exp\left(\frac{9}{4}\,ar(x)\right),\\
   |\nabla\nabla\Rm|(x,0)&\leq{}
          C\exp\left(\frac{27}{8}\,ar(x)\right)\ ,
\end{align*}
etc., here all the constants $C$ depend only on $n$. \\

We need to choose $D$ to be large enough in Lemma
\ref{lem:NW2007-(2.3)}. Clearly the lower bound for $f(x)$ would
then be $\frac{D}{2\rho}r(x)$ for large $r(x)$.

Now we turn to our integral identities. Fix a point $O\in{}M$, and
let $B_{r}=B_{r}(O)$ be the ball centered at $O$ with radius $r>0$
in $M$. Let $X_{k}=R_{lkjp}f_{p}R_{jl}e^{-f}$, then
\begin{equation*}
   R_{lkjp}f_{p}\nabla_{k}R_{jl}e^{-f}=
   \nabla_{k}X_{k}-R_{lkjp}f_{pk}R_{jl}e^{-f}\ ,
\end{equation*}
where we use the identity \eqref{eq:cao2007-(2.2)}. Hence we have
\begin{equation}\label{eq:IBP-(01a)}
   \int_{B_{r}}R_{lkjp}f_{p}\nabla_{k}R_{jl}e^{-f}d\mu=
   \int_{\partial{}B_{r}}\inner{X}{\nu}dA-
   \int_{B_{r}}R_{lkjp}f_{pk}R_{jl}e^{-f}d\mu\ ,
\end{equation}
where $\nu$ is the unit normal vector field on $\partial{}B_{r}$
and $dA$ is the induced measure on $\partial{}B_{r}$. We claim
that $\int_{\partial{}B_{r}}\inner{X}{\nu}dA\to{}0$ as
$r\to\infty$ and the other integrals in \eqref{eq:IBP-(01a)} are
finite as $r\to\infty$. In fact since ${\rm Ric}\geq -Kg$, by
Bishop-Gromov Comparison Theorem (cf. \cite{SY1994}), which
gives linear exponential growth control of the volume form, as
$r\to\infty$,
\begin{align*}
   \left|\int_{\partial{}B_{r}}\inner{X}{\nu}dA\right|
     &\leq{}C\int_{\partial{}B_{r}}|\Rm|\cdot|\nabla
     {}f|\cdot|\Ric|e^{-f}dA\\
     &\leq{}C\int_{\partial{}B_{1}}\exp\left[-\frac{D}
     {2\rho}\,r+ (2a+C)r\right]dA_1\to{}0,
\end{align*}
where $D$ is chosen to be large enough. For the other
integrals, still using Bishop-Gromov Comparison Theorem,
as for any $r$, we have
\begin{align*}
   \left|\int_{B_{r}\setminus{}B_{1}}
   R_{lkjp}f_{p}\nabla_{k}R_{jl}e^{-f}d\mu\right|
     \leq&\,\int_{B_{r}\setminus{}B_{1}}|\Rm|\cdot|\nabla{}f|\cdot|\nabla\Ric|
             e^{-f}d\mu\\
     \leq&\,C\int_{1}^{r}\exp\left[-\frac{D}{2\rho}\,t+
             \left(\frac{13}{4}\,a+C\right)t\right]dt<\infty,
\end{align*}
and
\begin{align*}
   \left|\int_{B_{r}\setminus{}B_{1}}R_{lkjp}f_{pk}R_{jl}e^{-f}d\mu\right|
    \leq&\,\int_{B_{r}\setminus{}B_{1}}|\Rm|\cdot(|\Ric|+n\rho)\cdot|\Ric|
           e^{-f}d\mu\\
    \leq&\,C\int_{1}^{r}\exp\left[-\frac{D}{2\rho}\,t+
           (3a+C)t\right]dt<\infty.
\end{align*}
This leads to
$$\left|\int_{B_{r}}R_{lkjp}f_{p}\nabla_{k}R_{jl}e^{-f}d\mu\right|<\infty
\;\; \mbox{and} \;\;
\left|\int_{B_{r}}R_{lkjp}f_{pk}R_{jl}e^{-f}d\mu\right|
   <\infty\ .$$
So we verify that \eqref{eq:IBP-(01)} is valid.
\end{proof}

Now we will finish the proof of Theorem \ref{main theorem 02}
using the above integral identity.

\begin{proof}[Proof of Theorem \ref{main theorem 02}]
Since $f(x)\leq Cr(x)^2$ for $x$ with $r(x)$ large enough, and $R+|\nabla f|^2
=2\rho f$, so we have $$R\leq 2\rho f\leq Cr^2.$$
With
$(M,g)$ is locally conformally flat and the Ricci curvature is bounded from
below,
it implies that $|\Ric|\leq Cr^2$. So the condition of Theorem 1 is satisfied.

If $(M,g)$ is
not flat, following
the proof of Theorem \ref{main theorem 01}, we have
\begin{equation*}
      \frac{1}{n}\int_{M}|\nabla{}R|^{2}e^{-f}d\mu \leq
      \frac{1}{2(n-1)}\int_{M}|\nabla{}R|^{2}e^{-f}d\mu \ .
\end{equation*}
Hence we have $\nabla{}R=0$ and $\nabla\Ric=0$. Since
$(M,g)$ is locally conformally flat, we conclude that $\nabla
\Rm=0$, i.e., $M$ is a locally symmetric space.

Using $\nabla\Ric=0$ and $\nabla_{i}(R_{ik}e^{-f})=0$, we have
\begin{equation*}
  0=(\nabla_{i}R_{ik})e^{-f}-R_{ik}f_{i}e^{-f}=
    -R_{ik}f_{i}e^{-f}\ ,
\end{equation*}
which implies $R_{ik}f_{i}=0$ (this identity also follows from the
identity $\nabla_{i}R=2R_{ij}f_{j}$ and the fact that $R$ is
constant). Moreover, we have
\begin{equation*}
   0=\nabla_{j}(R_{ik}f_{i})=R_{ik}f_{ij}\ ,
   \quad\forall\,j,k=1,\ldots,n\ ,
\end{equation*}
here we use the fact that $\nabla\Ric=0$. For any $x \in M$, we
diagonalize the Ricci curvature tensor $\Ric(x)$ in an orthonormal
frame (hence the Hessian $\nabla \nabla f$ is also diagonalized
because of the soliton equation (\ref{eq:grs})),
\begin{equation*}
   R_{ij}=\begin{pmatrix}
             \lambda_{1} & & & \\
                 & \lambda_{2} & & \\
                 & & \ddots & \\
                 & & & \lambda_{n}
          \end{pmatrix}\ .
\end{equation*}
So we arrive at $\lambda_{i}\cdot{}f_{ii}=0$ for each
$i\in\{1,\ldots,n\}$. This implies that, for each $i$, either
\begin{equation*}
   \lambda_{i}=0,
\end{equation*}
or
\begin{equation*}
   f_{ii}=0\ \ \text{but}\ \ \lambda_{i}\ne{}0.
\end{equation*}
 When $f_{ii}=0$, we get $\lambda_{i}=\rho$ by the soliton equation
 \eqref{eq:grs}.
Therefore without loss of generality, we may assume
\begin{equation*}
   \lambda_{i}=\begin{cases}
                  0\ , & 1\leq{}i\leq{}m\ ,\\
                  \rho\ , & m+1\leq{}i\leq{}n\ ,
               \end{cases}
\end{equation*}
where $1\leq{}m\leq{}n$. Now since $(M,g)$ is locally conformally
flat, by equation \eqref{rmc}, we have
\begin{equation*}
   R_{ijij}=\frac{1}{(n-1)(n-2)}[(n-1)(\lambda_{i}+\lambda_{j})-
   (n-m)\rho]\ .
\end{equation*}
Using an identity of Berger (also see Lemma 4.1 in \cite{CXD2007})
and the fact that $\nabla\Ric=0$, we have
\begin{align*}
   0&=\sum_{i<j}R_{ijij}(\lambda_{i}-\lambda_{j})^{2}
       =\left(\sum_{i<j\leq{}m}+\sum_{i\leq{}m<j}+
        \sum_{m+1\leq{}i<j}\right)
        R_{ijij}(\lambda_{i}-\lambda_{j})^{2}\\
      &=\frac{1}{(n-1)(n-2)}\sum_{i\leq{}m<j}
        [(n-1)(\lambda_{i}+\lambda_{j})-(n-m)\rho]
        (\lambda_{i}-\lambda_{j})^{2}\\
      &=\frac{1}{(n-1)(n-2)}\,m(m-1)(n-m)\rho^{3}\ .
\end{align*}
Thus $m=0,1$ or $n$, then we get
\begin{itemize}
   \item $\lambda_{1}=\lambda_{2}=\cdots=\lambda_{n}=\rho$, or
   \item $\lambda_{1}=0$ and
         $\lambda_{2}=\lambda_{3}=\cdots=\lambda_{n}=\rho$, or
   \item $\lambda_{1}=\lambda_{2}=\cdots=\lambda_{n}=0$.
\end{itemize}
Since $R=\lambda_{1}+\cdots+\lambda_{n}$ is constant on $M$, the
above result is global (i.e. if the eigenvalues of the Ricci
curvature tensor $\lambda_{1},\ldots,\lambda_{n}$ belong to one of
the above three cases at one point, then they belong to the same
case at any other points).

For the first case, $R=n\rho>0$, and we have
\begin{equation*}
   R_{ij}=\frac{R}{n}\,g_{ij}=\rho g_{ij} >0\ ,
\end{equation*}
so $(M,g)$ is compact, but we assume $(M,g)$ is complete
noncompact, this is a contradiction.

For the second case, $R=(n-1)\rho>0$, and we have
\begin{equation*}
   R_{1i}=R_{i1}=0\ ,\quad{}1\leq{}i\leq{}n\ ,
\end{equation*}
and
\begin{equation*}
   R_{jk}=\frac{R}{n-1}\,g_{jk}\ ,\quad{}2\leq{}j,k\leq{}n\ .
\end{equation*}

For the third case, we have $R_{ij}=0$ and $R=0$, so
$R_{ijkl}=0$, i.e., $(M,g)$ is flat. This also yields a contradiction.

Hence we can conclude that the sectional
curvature
\begin{equation*}
   K(e_{i},e_{j})=\frac{R_{ijij}}{g_{ii}g_{jj}-g_{ij}^{2}}
\end{equation*}
is nonnegative, so $(M,g)$ is a simply connected locally symmetric
space of nonnegative sectional curvature, then we prove the
statements of this theorem by using Theorem 10.1.1 in
\cite{W1984}.
\end{proof}


\bibliographystyle{alpha}
\bibliography{solitonbio}

\end{document}